\documentclass[11pt]{amsart}
\usepackage{hyperref}
\usepackage[top=2.7cm,left=2.8cm,right=2.8cm,bottom=2.6cm]{geometry}
\usepackage[dvipsnames]{xcolor}
\hypersetup{colorlinks=true,citecolor=NavyBlue,linkcolor=Brown,urlcolor=Orange}

\usepackage{amssymb,comment,paralist,xspace,graphicx,url,amscd,euscript, mathrsfs,stmaryrd,epic,eepic,color}
\usepackage{tikz-cd}
\usepackage[all]{xy}
\usepackage{amsthm}
\usepackage{amsmath}
\usepackage{amsfonts}
\usepackage{dsfont}

\usepackage{enumerate}
\usepackage{tikz}
\usetikzlibrary{shapes,arrows,decorations.pathreplacing,backgrounds,positioning,fit,matrix}
\tikzstyle{block} = [rectangle,draw,fill=blue!20,text width=5em,text centered,rounded corners,minimum height=4em]
\usepackage{float}

\usepackage[style=alphabetic,giveninits=true,backref=true]{biblatex}
\newtheorem*{namedtheorem}{\theoremname}
\newcommand{\theoremname}{testing}

\newtheorem{theorem}{Theorem}[section]

\newtheorem{proposition-definition}[theorem]{Proposition-Definition}
\newtheorem{corollary}[theorem]{Corollary}

\theoremstyle{definition}
\newtheorem{definition}[theorem]{Definition}

\newtheorem{remark}[theorem]{Remark}

\theoremstyle{remark}

\newcommand\cM{\mathcal{M}}

\newcommand\PP{\mathbb{P}}

\newcommand\ZZ{\mathbb{Z}}

\newcommand\rH{\mathrm{H}}

\title{Projective Equivalence of Smooth Hypersurfaces via Cyclic Covers}

\author{Zhiyuan Li}
\address{Shanghai Center for Mathematical Science, Shanghai, China.}
\email{\url{zhiyuan\_li@fudan.edu.cn}}

\author{Zhichao  Tang}
\address{Shanghai Center for Mathematical Science, Shanghai, China.}
\email{\url{zctang19@fudan.edu.cn}}
\date{October 2025}


\begin{document}

\begin{abstract}
    In this paper,  we prove that for any smooth hypersurface $Y$ of degree $d$ in $\mathbb{P}^{n+1}_k$, the cyclic $d$-fold cover $\widetilde{Y} \to \mathbb{P}^{n+1}_k$ branched along $Y$ completely characterizes $Y$ up to projective equivalence. This solves a question asked by Huybrechts in \cite[\S \textbf{1}.5.6]{Huy23}.
\end{abstract}

\maketitle

\section{Introduction}

Let \( k \) be an algebraically closed field of characteristic \( 0 \). Let \(\cM_d^n\) denote the moduli space of smooth hypersurfaces of degree \( d \) in \(\PP^{n+1}_k\). There is a natural map  
\begin{equation}\label{eq:moduli}  
    \Psi_d^n: \cM_{d}^n \to \cM_{d}^{n+1}  
\end{equation}  
defined by sending a hypersurface \( Y \subseteq \PP^{n+1}_k \) to its associated \( d \)-cyclic covering \(\widetilde{Y} \to \PP^{n+1}_k\) branched over \( Y \).  
Explicitly, if \( Y = V(F) \) for some homogeneous polynomial \( F \in k[x_0, \ldots, x_{n+1}] \), then \(\widetilde{Y} = V(x_{n+2}^d - F) \subseteq \PP^{n+2}_k\),  
equipped with a canonical \(\rho\)-action, where \(\rho\) is a primitive \(d\)-th root of unity. This action is given by \(x_i \mapsto x_i\) for \(0 \leq i \leq n+1\) and \(x_{n+2} \mapsto \rho \cdot x_{n+2}\), realizing \(\rho\) as the generator of the Galois group of the covering.  

This construction has proven instrumental in relating moduli spaces of cubic hypersurfaces across different dimensions. A fundamental question posed by Huybrechts \cite[Remark \textbf{1}.5.22]{Huy23} is whether the map \(\Psi_d^n\) in \eqref{eq:moduli} is injective. For \( n \geq 2 \), this is equivalent to asking whether the isomorphism class of \( Y \) is uniquely determined by that of its cyclic cover \(\widetilde{Y}\).  

For generic hypersurfaces of dimension \( n \geq 2 \), injectivity can possibly be established using the Hessian determinant (see \cite[Theorem 4.6]{MR2655320} for the cubic case). Leveraging this, Allcock, Carlson, and Toledo established the link between the moduli space of cubic surfaces (respectively, threefolds) with the moduli space of cubic threefolds (respectively, fourfolds) in \cite{ACT02, ACT11}.  

In this paper, we resolve Huybrechts' question in full generality. Our method makes use of the structure result of outer Galois points, yielding a very simple proof of the following:  
\begin{theorem}\label{mainthm}  
    The map \(\Psi_d^n: \cM_{d}^n \to \cM_{d}^{n+1}\) is injective for all \( n \geq 2 \) and \( d \geq 2 \).  
\end{theorem}  

As a consequence, we refine the global Torelli theorem for cubic surfaces and threefolds via cyclic covers (cf. \cite[Theorem \textbf{4}.4.9]{Huy23}): \( Y \cong Y' \) if and only if there exists a Hodge isometry \(\rH^m(\widetilde{Y}, \ZZ) \cong \rH^m(\widetilde{Y}', \ZZ)\), where \( m = \dim \widetilde{Y} \). Crucially, Theorem \ref{mainthm} eliminates the need for the \(\ZZ[\rho]\)-equivariance assumptions.

\section{Cyclic cover and outer Galois points}

We recall the connection between cyclic covers and outer Galois points for smooth hypersurfaces, following Yoshihara \cite{Yo03}.

\begin{definition}
    Let \(X = V(F) \subseteq \PP^{n+1}_k\) be a smooth hypersurface of degree \(d\). 
    A point \(P \in \PP^{n+1}_k \setminus X\) is an \emph{outer Galois point} if the projection \(\pi_P \colon X \to H\) from \(P\) to a hyperplane \(H \cong \PP^n_k\) induces a Galois extension \(k(X)/k(H)\),
    where the field extension is independent of  the choice of $H$ with $P\notin H$.
    The Galois group is denoted \(G_P = \operatorname{Gal}(k(X)/k(H))\).
\end{definition}

By \cite[Theorem 5]{Yo03}, if $X$ has an outer Galois point $P$, \(G_P\) is cyclic of order \(d\), and \(\pi_P\) is a cyclic cover branched along a degree \(d\) hypersurface in $\PP^n_k$. 
After coordinate change, \(X\) admits the normalized form \[F = x_{n+1}^d + G(x_0, \dots, x_{n})\] with \(\deg G = d\). 
Conversely, any hypersurface of the form \( V(x_{n+1}^d + G) \subseteq \PP^{n+1}_k\) has \([0:\cdots:0:1]\) as an outer Galois point.

Let \(\delta(X)\) denote the number of outer Galois points. The following structural result is the key ingredient:

\begin{theorem}[{cf.~\cite[Proposition 11]{Yo03}}]\label{thm:structure} 
Let \(X\) be a smooth hypersurface in \(\PP^{n+1}_k\) of degree \(d \geq 3\). Then 
  \begin{enumerate}[(i)]
      \item \(0\leq \delta(X) \leq n+2\).
      \item If \(\delta(X)> 0\),  \(X\) is projectively equivalent to 
        \begin{equation}
            V\left( x_{n+1}^d + x_{n}^d + \cdots + x_{n+1-r}^d + G(x_{0}, \ldots, x_{n-r}) \right)\,,
        \end{equation}
        where  $r=\delta(X)-1$  and \(G\) is homogeneous of degree \(d\). 
  \end{enumerate}
\end{theorem}

\begin{proof}
 The original statement is over complex numbers $\mathbb{C}$. Its proof can be extended to any algebraically closed field of characteristic zero.
 We outline the key steps:
\vspace{.1cm}

\noindent \textbf{(i) Bound on \(\delta(X)\):}
\begin{itemize}
    \item For any outer Galois point \(P\), the projection \(\pi_P: X \rightarrow H\) induces a Galois extension \(K/K_P\) of degree \(d\), where \(K = k(X)\) and \(K_P = k(H)\).
    \item By the analog of \cite[Lemma 1]{Yo03}, elements of \(\mathrm{Gal}(K/K_P)\) induce projective automorphisms of \(X\).
   \item Outer Galois points are \emph{independent} (i.e., no three are collinear) by an inductive argument using hyperplane sections (cf. \cite[Lemma 13]{Yo03}). 
    \item Since independent points in \(\mathbb{P}^{n+1}\) lie in general position, we obtain \(\delta(X) \leq n+2\).
\end{itemize}
\vspace{.1cm}

\noindent \textbf{(ii) Structure when \(\delta(X) > 0\):}
Let \(\delta(X) = r + 1\) with \(r \geq 0\), and let \(P_0, \dots, P_r\) be distinct outer Galois points.
\begin{itemize}
    \item By independence, choose coordinates \([x_0 : \cdots : x_{n+1}]\) such that \(P_i\) is the \(i\)-th coordinate point.
    \item For each \(P_i\), the Galois group \(G_{P_i}\) is cyclic of order \(d\) (by the analog of \cite[Theorem 5]{Yo03}). 
    As $k$ is algebraically closed, primitive \(d\)-th roots of unity exist in \(k\), and a generator \(\sigma_i\) acts diagonally as:
        \[
        \sigma_i = \operatorname{diag}[\rho, \dots, \rho, 1, \rho, \dots, \rho]\,,
        \]
        where \(1\) is in the \(i\)-th position and \(\rho\) is a primitive \(d\)-th root of unity.
    \item The defining polynomial \(F\) of $X$ must be invariant under each \(\sigma_i\) (up to scalars). This forces \(F\) to contain terms \(x_i^d\) for \(i = n+1-r, \dots, n+1\) (after reindexing), with no mixed terms involving these variables.
    \item Thus, \(F\) takes the form:
        \[
        F = x_{n+1}^d + x_n^d + \cdots + x_{n+1-r}^d + G(x_0, \dots, x_{n-r})\,,
        \]
        where \(G\) is homogeneous of degree \(d\). Smoothness of \(X\) ensures the coefficients of \(x_i^d\) are non-zero and \(G\) is general.
    \item Any such \(X\) is projectively equivalent to this normal form.
\end{itemize}
\end{proof}

\begin{remark}
Most of \cite{Yo03} assumes $d \geq 4$ since it deals with both inner and outer Galois points. However, for outer Galois points specifically, the weaker condition $d \geq 3$ suffices (see \cite[Lemma 13]{Yo03} and \cite[Theorem 1.2]{zbMATH06268338}).
    This distinction arises because
        \begin{itemize}
            \item Projection from an inner point gives a degree $d-1$ extension;
            \item Projection from an outer point gives a degree $d$ extension.
        \end{itemize}
    Since quadratic extensions are always Galois, non-triviality requires extensions of degree $\geq 3$. Thus, inner Galois points require $d \geq 4$, while outer Galois points require only $d \geq 3$.
\end{remark}

\section{Proof of the main theorem}

When $d<3$, this is trivial. Let us assume that $d\geq 3$. 
Let \(Y_1 = V(F_1)\) and \(Y_2 = V(F_2)\) be smooth hypersurfaces in \(\PP^{n+1}_k\).  
Suppose that the corresponding cyclic coverings
\(X_1 = V(x_{n+2}^d - F_1)\) and \(X_2 = V(x_{n+2}^d -F_2)\) in \(\PP^{n+2}_k\) are projectively equivalent. 
Let \(f \colon \PP^{n+2}_k \to \PP^{n+2}_k\) be a linear automorphism with \(f(X_2) = X_1\). 

Denote by \(P_0, \dots, P_r\) the outer Galois points of \(X_1\). By Theorem \ref{thm:structure}, after projective equivalence we may assume 
\[
    F_1 =  G(x_0, \dots, x_{n-r})+\sum_{i=0}^r x_{n+1-i}^d 
\]
and hence the outer Galois points of $X_1$ are \(P_i = [0:\cdots:0:1:\underbrace{0:\cdots:0}_{i}]\). 
Since projective equivalence preserves outer Galois points, \(f\) maps the outer Galois points of \(X_2\) bijectively to those of \(X_1\). 
In particular, the standard outer Galois point \([0:\cdots:0:1]\) of \(X_2\) is mapped to some \(P_i\). 
The transposition \(\sigma\) swapping coordinates \(x_{n+2-i}\) and \(x_{n+2}\) satisfies:
\begin{align*}
    \sigma(P_0) &= P_i\,, \\
    \sigma(X_1) &= X_1\,.
\end{align*}
Composing \(f\) with \(\sigma\) yields a projective equivalence \(f' \colon X_2 \xrightarrow{\sim} X_1\) with \(f'(P_0) = P_0\). 

Since \(f'\) fixes \(P_0 \),  \(f'\) must have the form:
\[
f'([x_0 : \cdots : x_{n+2}]) = [\ell_0(\mathbf{x'}) : \cdots : \ell_{n+1}(\mathbf{x'}) :  x_{n+2} + \ell_{n+2}(\mathbf{x'})]\,,
\]
where \(\mathbf{x'} = (x_0,\dots,x_{n+1})\), and each $\ell_{i}$ is a linear form.
The projective equivalence requires:
\[
x_{n+2}^d + F_2(\mathbf{x'}) = (x_{n+2} + \ell_{n+2}(\mathbf{x'}))^d +  F_1(f'(\mathbf{x'}))\,.
\]
This forces $\ell_{n+2}=0$ and $f'$ is block diagonal. 
It follows that $F_1$ is projectively equivalent to $F_2$. 

\begin{corollary}\label{cor:cyclic-cover}
    Let \(\pi_1: X_1 \to \PP^{n+1}_k\) and \(\pi_2: X_2 \to \PP^{n+1}_k\) be \(d\)-cyclic coverings branched along smooth hypersurfaces \(Y_1\) and \(Y_2\) of degree \(d\). 
    If \(n \geq 2\), then \(X_1 \cong X_2\) if and only if \(Y_1 \cong Y_2\).
    \end{corollary}
\begin{proof}
    Each \(X_i\) embeds as a hypersurface in \(\PP^{n+2}_k\).  After coordinate change, we normalize to
    \[
    X_i = V\left( x_{n+2}^d + F_i(x_0,\ldots,x_{n+1}) \right)\subseteq \PP^{n+2}_k\,,
    \]
    where the branching locus satisfies \(Y_i \cong V(F_i) \subseteq \PP^{n+1}_k\).

    If \(X_1 \cong X_2\), since \(X_i\) has Picard number one (as \(n \geq 2\)),  \(V(x_{n+2}^d + F_1)\) is projectively equivalent to \(V(x_{n+2}^d + F_2)\). 
    By Theorem \ref{mainthm}, this implies \(V(F_1) \cong V(F_2)\), hence \(Y_1 \cong Y_2\). 
    For the other direction, this is obvious. 
\end{proof}

\section{Remark on positive characteristics}

We now discuss extensions to positive characteristics. Let $\mathrm{char}(k) = p > 0$ with $p \nmid d$. The theory of $d$-cyclic covers and outer Galois points largely parallels the characteristic zero case. By \cite[Theorem 1.1]{arXiv:2509.20945}, if $X$ has an outer Galois point $P$, then $G_P$ is cyclic of order $d$ and $\pi_P$ is a cyclic cover branched along a degree $d$ hypersurface in $\PP^n_k$.

However, Theorem \ref{thm:structure} may fail in positive characteristics. For instance, when $d-1$ is a power of $p$, smooth hypersurfaces can exhibit more Galois points than expected; see \cite[Theorem 1]{zbMATH05126212} and \cite[\S 4]{zbMATH06268338}. 

The proof strategy for Theorem \ref{thm:structure} remains largely valid except for the independence argument, which proceeds by induction on dimensions and relies on the following two facts:
\begin{itemize}
    \item[(i)] The base case when $\dim X = 1$;
    \item[(ii)] A \textbf{Bertini-type theorem}: For any two outer Galois points $P$, $Q$, the linear system $\Lambda$ of hyperplanes containing $P$ and $Q$ contains a member $H$ such that $H \cap X$ is smooth.
\end{itemize}

For (i), the result also holds when $p > 2$ and $d-1$ is not a power of $p$, by the work of Fukasawa \cite{MR3090633} which classifies Galois points of smooth plane curves over algebraically closed fields of arbitrary characteristic. We expect Theorem \ref{mainthm} still holds under such assumptions.

The main problem is (ii). In positive characteristics, it remains unknown whether such a Bertini-type theorem holds for the linear system $\Lambda$. The optimal known result is \cite[Theorem 1.3]{zbMATH06268338}, which guarantees that a general hyperplane passing through a Galois point $P$ intersects $X$ smoothly; 
and if the intersection is smooth, $P$ will be a Galois point of the intersection $H \cap X$, which is of lower dimension and one may perform induction.

\section*{Declaration}

\subsection*{Funding}
Z.~Li is supported by the NSFC grants (No.~12171090 and No.~12425105) and the Shanghai Pilot Program for Basic Research (No.~21TQ00). Z.~Li is also a member of LMNS.
Z.~Tang is supported by the NSFC grant (No.~12121001).

\subsection*{Conflict of interest}
The authors have no competing interests to declare that are relevant to the content of this article.

\printbibliography

\end{document}